\documentclass[12pt,a4paper]{article}
\usepackage{amssymb,amsmath,amsthm}
\usepackage{graphicx}
\usepackage{graphics}
\usepackage{enumerate}

\newcommand\cA{{\mathcal A}}
\newcommand\cB{{\mathcal B}}

\newcommand\cG{{\mathcal G}}

\newcommand{\Vect}[1]{\underline{#1}}

\theoremstyle{plain}
\newtheorem{theorem}{Theorem}[section]
\newtheorem{lemma}[theorem]{Lemma}
\newtheorem{cor}[theorem]{Corollary}

\theoremstyle{definition}

\newtheorem{claim}[theorem]{Claim}

\newtheorem{obs}[theorem]{Observation}

\newcommand\lref[1]{Lemma~\ref{lem:#1}}
\newcommand\tref[1]{Theorem~\ref{thm:#1}}
\newcommand\cref[1]{Corollary~\ref{cor:#1}}

\newcommand\sref[1]{Section~\ref{sec:#1}}

\begin{document}

\title{Identifying codes and searching with balls in graphs}

\author{Younjin Kim\thanks{Department of Mathematical Sciences, KAIST, 291 Daehak-ro Yuseong-gu Daejeon, 305-701 South Korea. Email: younjin@kaist.ac.kr. The author is supported by Basic Science Research Program through the
National Research Foundation of Korea(NRF) funded by the Ministry of
Science, ICT \& Future Planning (2011-0011653) }
 \and Mohit Kumbhat\thanks{Sungkyunkwan Unversity, Suwon, S.Korea 440-746. Email:mohitkumbhat@gmail.com } \and
Zolt\'an L\'or\'ant Nagy\thanks{Alfr\'ed R\'enyi Institute of Mathematics, P.O.B. 127, Budapest H-1364, Hungary and E\"{o}tv\"{o}s Lor\'{a}nd University, Department of Computer Science, H-1117 Budapest P\'{a}zm\'{a}ny P\'{e}ter s\'{e}t\'{a}ny 1/C. Email: nagyzoltanlorant@gmail.com. The author was supported by the Hungarian National Foundation for Scientific Research (OTKA), Grant no. K 81310.} \and Bal\'azs Patk\'os\thanks{MTA--ELTE Geometric and Algebraic Combinatorics Research Group, H--1117 Budapest, P\'azm\'any P.\ s\'et\'any 1/C, Hungary. Email: patkosb@cs.elte.hu. Research is supported by and
 the J\'anos Bolyai Research Scholarship of the Hungarian Academy of Sciences.} \and Alexey Pokrovskiy \thanks{Methods for Discrete Structures, Berlin. Email: apokrovskiy@zedat.fu-berlin.de} \and M\'at\'e Vizer\thanks{Alfr\'ed R\'enyi Institute of Mathematics, P.O.B. 127, Budapest H-1364, Hungary. Email: vizermate@gmail.com Research supported by
    Hungarian National Scientific Fund, grant number: 83726.} }

\maketitle

\begin{abstract}
Given a graph $G$ and a positive integer $R$ we address the following combinatorial search theoretic problem: What is the minimum number of queries
of the form ``does an unknown vertex $v \in V(G)$ belong to the ball of radius $r$ around $u$?" with $u \in V(G)$ and $r\le R$ that is needed to determine $v$. We consider both the adaptive case when the $j$th query might depend on the answers to the previous queries and the non-adaptive case when all queries must be made at once. We obtain bounds on the minimum number of queries for hypercubes, the Erd\H os-R\'enyi random graphs and graphs of bounded maximum degree .
\end{abstract}


\section{Introduction}
Combinatorial search theory is concerned with problems of the following type ``given a finite set $S$ and an unknown element $x\in S$, find $x$ as quickly as possible''. In order to find $x$, we are allowed to ask questions of the type ``is $x$ contained in a subset $B\subseteq S$?'' for various $B\subseteq S$. There are many real-world search problems which are of this type. For example given a set of blood samples, we might want to identify one which is infected. Or given a building, we might want to determine which room a person is hiding in.

Notice that to guarantee finding the unknown element $x$ in a set $S$ we always need to ask at least $\log_2 |S|$ questions. Indeed, to be able to distinguish between two elements $x, x'\in S$, we must at some point query a set $B\subseteq S$ such that $x\in B$ and $x'\not\in B$. However if we ask less than  $\log_2 |S|$ questions, then by the Pigeonhole Principle, there would be two  elements $x$ and $x'\in S$ which receive the same sequence of yes/no answers.

The lower bound  $\log_2 |S|$ on the number of questions can be tight.
In fact it is always possible to find $x$ using $\lceil\log_2 |S|\rceil$ questions. Indeed first we divide $S$ is half and ask which half $x$ is in. Then we divide the half containing $x$ in half again, and find out which one $x$ is in. Proceeding this way, we reduce the subset of $S$ which $x$ may be in by half at every step, and hence we find $x$ after at most $\lceil\log_2 |S|\rceil$ queries.  This is the so called \emph{halving search}. 

Notice that in the above search the set we query at each step depends on the answers we got in the previous steps. We say that a search is \emph{adaptive} if a question is allowed to depend on answers to previous questions. In practice we often would like to make all our queries at the same time (for example if each individual query takes a long time, it would save time to make queries simultaneously). We say that a search is \emph{non-adaptive} if the queries do not depend on answers obtained from previous ones. Just like in the adaptive case, there exist non-adaptive searches using the optimal  $\log_2 |S|$ number of queries. For example if $S$ is the set of $0-1$ sequences of length $n$, then we can find $x=(x_1, \dots, x_n)$ using the $n=\log_2 |S|$ queries ``is $x_i$ equal to zero or one?''  

Note that the answers to a set $Q$ of query sets surely determine the unknown element $x$ if and only if for any pair of elements $s_1,s_2 \in S$ there exists a query $B \in Q$ with $|B \cap \{s_1,s_2\}|=1$. In this case we say that $B$ \textit{separates} $s_1$ and $s_2$, and $Q$ is said to be a \textit{separating system} or said to possess the \textit{separating property}.

In this paper we study a search problem when the set $S$ is a metric space, and we are only allowed to query sets $B$ which are balls in the metric space. We call this the \emph{ball search}. There are natural search problems of this type. For example if we want to track someone moving through a building, we might set up sensors at various points in the building which tell us if the person is within a certain range. This problem is inherently non-adaptive since the person is moving (and it would usually be inpractical to change the location of the sensors as the person moves). Another example of a ball search is to locate some object in the night sky using a telescope. This problem is inherently adaptive since once we capture the object with the telescope, we would want to zoom in in order to get a better idea of its location.

Ball searches have been studied before, under the name of \emph{identifying codes on graphs}. 
A ball $B_r(v)$ of radius $r$ at vertex $v$ in a graph $G$ is the set of vertices in $G$ that have distance at most $r$ from $v$. An $(r,\le l)$-identifying code of the graph $G$ is a set $C$ of vertices such that for every pair of distinct subsets $X,Y \subset V(G)$ with $|X|,|Y| \le l$ the sets $B_r(X)\cap C$ and $B_r(Y)\cap C$ are different and non-empty. Notice that given a graph with an $(r,\le 1)$-identifying code $C$, we can always find an unknown vertex using the nonadaptive search querying the balls $\{B_r(c):c\in C\}$. Identifying codes were introduced by Karpovsky, Chakrabarty, and Levitin \cite{KCL} and the problem of determining the minimum size $i^{(l)}_r(G)$ of an identifying code in $G$ has since attracted the attention of many researchers  (for a full bibliography see \cite{L}). An adaptive version of identifying codes was introduced by Moncel in his doctoral thesis \cite{M} and in \cite{BGLM}. The minimum number of queries needed in the adaptive case is denoted by $a^{(l)}_r(G)$.

We will study a slight variant of identifying codes in which the restriction that the balls in the search must cover \emph{all} the vertices of $G$ is omitted. The reason we study this variant is that it seems to have links to a  natural combinatorial object (namely the Fano Plane. See Section 2 for details). We denote by $M(G,r)$ the minimum number of queries needed to nonadaptively find an unknown vertex $v\in G$ using balls of radius $r$.  Similarly we denote by $A(G,r)$ the minimum number of queries needed to adaptively find an unknown vertex $v\in G$ using balls of radius $r$. Notice that we always have $i^{(1)}_r(G)\leq  M(G,r)\leq  i^{(1)}_r(G)+1$ and $a^{(1)}_r(G)\leq  A(G,r)\leq  a^{(1)}_r(G)+1$. Therefore, since we will mostly be interested in asymptotic estimates, it usually won't matter whether we study  the quantity $M(G,r)$ or $i^{(1)}_r(G)$.

Notice that in the above ball searches we specified that we could only query balls of radius \emph{exactly} $r$. With this restriction it is not always possible to find an unknown vertex $v$---for example if $r$ is bigger than the diameter of $G$ (the maximum distance between two vertices in $G$), then querying balls of radius $r$ gives no information about the location of $v$.
At the 5th Eml\'ekt\'abla Workshop, Gyula Katona asked what happens \cite{eml} if balls of radius \textit{at most} $r$ should be allowed as queries. We'll denote by  $M(G,\leq r)$ and $A(G,\leq r)$ the minimum number of queries to  find an unknown vertex $v\in G$ using balls of radius $r$ in the nonadaptive and adaptive searches respectively. We'll denote by $M(G)$ and $A(G)$ the minimum number of queries to  find an unknown vertex $v\in G$ using balls (of any radius) in the nonadaptive and adaptive searches respectively.   Notice that since balls of radius zero are just single vertices in $G$, we always have $M(G,\leq r), A(G,\leq r)\leq |G|$. Also, we trivially have $M(G,\leq r)\leq M(G, r)$ and $A(G,\leq r)\leq A(G, r)$. The quantities $M(G,\leq r)$ and $A(G,\leq r)$  are monotonic in the sense that $M(G)\leq M(G,\leq r+1)\leq M(G,\leq r)$ and $A(G)\leq A(G,\leq r+1)\leq A(G,\leq r)$ hold for all $r$.

The first metric spaces in which we will study the ball search are hypercubes.
Let $Q_n$ be the $n$ dimensional hypercube---the set of  all $0-1$ vectors of length $n$. We define a graph on $Q_n$  by placing an edge between two vectors whenever they differ in exactly one entry. In this graph the distance $d(\underline{u},\underline{v})$ between two vertices is exactly the \emph{Hamming distance} between $\underline{u}$ and $\underline{v}$ i.e. the number of entries on which $\underline{u}$ and $\underline{v}$ differ.
If $\underline{c}$ is an element of $Q_n$  and $r$ is a non-negative integer, then $B(\underline{c}, r)$ denotes the ball of center $\underline{c}$ and radius $r$, that is, $B(\underline{c}, r):=\{\underline{v}: \underline{v}\in Q_n,  d(\underline{c},\underline{v})\leq r\}.$ The most recent upper and lower bounds on $i^{(l)}_r(Q_n)$ were proved in \cite{CCHL,ELR,EJLR1,EJLR2,JL}. Adaptive identification in $Q_n$ was studied by Junnila \cite{J} who obtained lower and upper bounds on $a^{(1)}_1(Q_n)$. To be able to state our results let $K(n,r)$ denote the minimum number of balls of radius $r$ that cover all vertices of $Q_n$. The centers of such balls are said to form \textit{covering codes} in $Q_n$ \cite{CHLL}.

\begin{theorem} \label{thm:main}
The functions $A(Q_n)$, $A(Q_n,\leq r)$ and $A(Q_n, r)$  satisfy the following inequalities:
\begin{enumerate}[(i)]
\item $n\leq A(Q_n) \leq n-1+ \left\lceil \log(n+1)\right\rceil$.
\item $K(n,r)-1\le A(Q_n,\leq r) \le K(n,r)+n-1+\lceil \log(r+1)\rceil$. There is a constant $C$ such that if $r\le n/2-C\sqrt{n}\log n$, then $A(Q_n,\leq r)=(1+o(1))K(n,r)$.
\item $K(n,r)-1\le A(Q_n,r) \le K(n,r)+\sum_{i=0}^r \binom{n}{r}$. There is a constant $C'$ such that if $r\le C'n$, then $A(Q_n, r)=(1+o(1))K(n,r)$.
\end{enumerate}
\end{theorem}

In Section 2, we will also study $M(Q_n)$ for small values of $n$. 

Then we turn our attention to the Erd\H os-R\'enyi random graph model $G(n,p)$. The asymptotics of $i^{(1)}_1(G(n,p))$ was determined by Frieze, Martin, Moncel, Ruszink\'o, and Smyth \cite{FMMRS} provided $p,1-p\ge 4 \log\log n/\log n$ holds.
Using results of Bollob\'as \cite{B} and Katona \cite{K} we obtain upper and lower bounds on $M(G(n,p),r)$ that differ by an $O(\ln n)$-factor  for any $r$ that is smaller than the diameter of $G(n,p)$ whenever
 $p=\Omega(n^{-1+\varepsilon})$. As the formulation of our theorem is somewhat technical, we defer its statement until \sref{rand}.

Finally, we consider graphs of bounded maximum degree. Results concerning the parameter $i^{(1)}_1(G)$ for such graphs, for random $d$-regular graphs and for some other graph classes defined by their degree sequences were obtained by Foucaud and Perarnau \cite{FP}. Let us denote by $\cG_{\Delta,n}$ the set of finite graphs on $n$ vertices with maximum degree at most $\Delta$.

\begin{theorem} \label{thm:maxdeg} For all connected graphs $G \in \cG_{\Delta,n}$ we have $A(G) = \Theta_\Delta(\log n)$.
\end{theorem}
Notice that the connectivity assumption cannot be ommited in the above theorem. Indeed, otherwise $G$ could be a union of disjoint cliques on $\Delta+1$ vertices. A search in this graph must ask at least one query in each component, and hence use at least $|G|/(\Delta+1)= O_{Delta}(n)$ queries.
Similarly, the bound on the maximal degree of $G$ cannot be removed since an optimal search on a star with $n$ vertices needs  $n-1$ queries.
Finally, note that Theorem~\ref{thm:maxdeg} cannot be strengthened to give a logarithmic bound on the length of a  non-adaptive search on bounded degree graphs. Indeed it is easy to show that $M(P_n)=\lceil \frac{n-1}{2}\rceil$, and so even for graphs with $\Delta=2$ , the nonadaptive search may need linearly many queries.

Let us observe the connection of non-adaptive ball search to another well studied parameter of graphs. A set $D$ of vertices of $G$ is called a \textit{resolving set} of $G$ if for any vertices $x,y \in V(G)$ there exists $z \in D$ such that $d(x,z)\neq d(y,z)$. The minimum size of a resolving set is called the \textit{metric dimension of $G$} (introduced in \cite{HM,S}) and is denoted by $\beta(G)$. Clearly, if a set of ball queries finds an unknown vertex $v$, then the centers of the balls form a resolving set in $G$ and thus we have $\beta(G) \le M(G)$. Also, if $D$ is a resolving set in $G$, then the query set $\{B_G(d,i): d \in D, 1\le i \le diam(G)\}$ finds the unknown vertex of $G$, where $diam(G)$ denotes the diameter of $G$. Consequently, we have $M(G) \le diam(G)\beta(G)$.

Throughout the paper ``$\log$'' stands for the logarithm of base 2, and ``$\ln$'' denotes the natural logarithm. 

\section{The hypercube}
\label{sec:hyp}

In this section we consider the adaptive search on $Q_n$. We find bounds on $A(Q_n)$ and $A(Q_n, \leq r)$.
We also find a code in $Q_7$ of balls of radius 3 that is in some sense better than the best previously known one. We start by proving \tref{main}.

Throughout this section we will denote the number of vertices in a ball of radius $r$ in $Q_n$ by $V(n,r)=\sum_{i=0}^r \binom nr$. Recall that that the following bound on $V(v,r)$ follows from the Chernoff Bound (see \cite{AS}) $|V(n,r)|\leq 2^n \exp(-(n-2r)/\sqrt{n})$.
\begin{proof}[Proof of \tref{main}:]
\begin{enumerate}[(i)]
\item
For the lower bound notice that the sequence of answers must be different for each vertex. Therefore there must be at least $\log(|Q_n|)=n$ questions.

For  the upper bound, let  $\underline{u}^*$ be the unknown element of $Q_n$.  We give an adaptive algorithm that finds  $\underline{u}^*$ and uses at most $n-1+ \left\lceil \log(n+1)\right\rceil$ queries.

The first part of our algorithm uses $\left\lceil \log(n+1)\right\rceil$ queries to determine the distance $d(\underline{0},\underline{u}^*)$. Notice that asking a ball $B(\underline{0}, r)$
is equivalent to asking whether we have $d(\underline{0},\underline{u}^*)\leq r$ or $d(\underline{0},\underline{u}^*)>r$. Note that $d(\underline{0},\underline{u}^*)>r$ can be anything in the set $\{0,1,\ldots n\}$. In the first query we ask the ball $B(\underline{0}, \lfloor n/2 \rfloor)$. Depending on whether we get a ``yes'' or ``no'' answer, we obtain that $d(\underline{0},\underline{u}^*)$ is either in $\{0,\ldots \lfloor n/2 \rfloor\}$ or in $\{\lfloor n/2 \rfloor+1, \dots,n\}$. Continuing in this fashion, always asking  $B(\underline{0}, r)$ where $r$ is the median of the possible (not already excluded) distance values, we see that  $\left\lceil \log(n+1)\right\rceil$ determine $d(\underline{0},\underline{u}^*)$.  Let us denote this distance by $d$.

To finish the algorithm we show that for every $1 \le i \le n$ one further query is enough to determine the value of $\underline{u}^*$'s $i$th coordinate. Indeed, $\underline{u}^* \in  B(\underline{e}_i, d-1)$ holds if and only if the $i$th coordinate of $\underline{u}^*$ is $1$. We ask the balls $B(\underline{e}_i, d-1)$ for $i = 1, \dots, n-1$ in order to determine $\underline{u}^*_1, \dots, \underline{u}^*_{n-1}$. Since $\underline{u}^*_n=d-\sum_{i=1}^{n-1} \underline{u}^*_i$, this is enough to determine~$\underline{u}^*$.

\item
For the lower bound, observe that if after asking balls $B_1,B_2,\dots,B_t$ we are able to determine the unknown vertex, then at most one vertex is not covered by the union of the balls.

For the upper bound, consider the following algorithm. First, let us ask a set of balls $\cB$ of radius $r$ that cover all vertices and with $|\cB|=K(n,r)$. After finding a ball $B \in \cB$ containing the marked vertex, we repeat the algorithm of part \textbf{(i)} with the modification that we know that the marked vertex is of Hamming distance at most $r$ from the center of $B$.  Therefore, to determine this distance exactly, we only need $\lceil\log (r+1)\rceil$ queries.

We now show that  there is a constant $C$ such that if $r\le n/2-C\sqrt{n}\log n$, then $A(Q_n,\leq r)=(1+o(1))K(n,r)$. The exact value of $K(n,r)$ is known
only for a short range of $n$ and $r$ and even determining the asymptotics of $K(n,r)$ is one of the major open problems of coding theory. For an introduction to covering codes, see the book \cite{CHLL}. Clearly $K(n,r) \ge 2^n/V(n,r)$ holds. It is known that for any fixed $r$, we have $K(n,r)=\Theta(2^n/V(n,r))$ (the best known constant is due to Krivelevich, Sudakov and Vu \cite{KSV}) and that for any function $r=r(n)$ we have $K(n,r)=O(n2^n/V(n,r))$.

By the above reasoning we obtained that to prove $A(Q_n,\leq r)=(1+o(1))K(n,r)$, it is sufficient to show that $n=o(2^n/V(n,r))$ holds. This holds since we have $2^n/V(n,r)\geq \exp((n-2r)/\sqrt{n})$ and $r\le n/2-C\sqrt{n}\log n$.

\item
The lower bound follows from part (ii). For the upper bound, proceed as we did in part (ii) to use $K(n,r)$ queries to find one ball of radius $r$ containing the marked vertex. After that we can eliminate one possible vertex by each further query thus finding $\Vect{u^*}$ with $K(n,r)+V(n,r)$ queries.

 The second part follows from the fact that if $r=r(n)=c_nn$ with $c_n\rightarrow 0$, then $\log V(n,r)/n \rightarrow 0$.
\end{enumerate}
\end{proof}



We finish this section with considering the special case $n=7$. There are many results on $i_r^{(1)}(Q_n)$ for small values of $r$ and $n$. In \cite{CCHL}, Charon, Cohen, Hudry, and Lobstein proved that $i^{(1)}_3(Q_7)=8$. This is true, but the lower bound is due to the fact that the definition of identifying code includes the condition that balls around the codewords should cover all vertices of the graph. On the other hand, no 7 of their 8 codewords possess the separation property. This can be checked easily as if they did, then for any two of them $B(\Vect{c_1},3)\cap B(\Vect{c_2},3),
B(\Vect{c_1},3)\setminus B(\Vect{c_2},3), B(\Vect{c_2},3)\setminus B(\Vect{c_1},3), V(Q_7) \setminus (B(\Vect{c_1},3)\cup B(\Vect{c_2},3))$ should have size 32 and consequently those 7 codewords should have pairwise Hamming distance 4. In the remainder of this section we show a construction of seven codewords with the separation property.

\begin{lemma}\label{Fano}
There exist vertices $\Vect{v_1}, \Vect{v_2},...,\Vect{v_7} \in V(Q_7)$ such that the balls $\{B(\Vect{v_i},3), i=1,2,\ldots,7\}$ separate $V(Q_7)$.
\end{lemma}

\begin{proof}

Consider the incidence matrix of the Fano plane $PG(2,2)$ consisting of $7$ rows and columns corresponding to its line and point set, see Figure 1.
Let $v_i$ be the $i$th row corresponding to the $i$th line in the finite geometry. We show that the following query set provides a separating ball system, implying the lemma:

\[ \{B(\underline{v_i}, 3): i=1\ldots 7  \} .\]

\begin{figure}[h!]\label{fanni}
\begin{center}
\includegraphics[width=3cm]{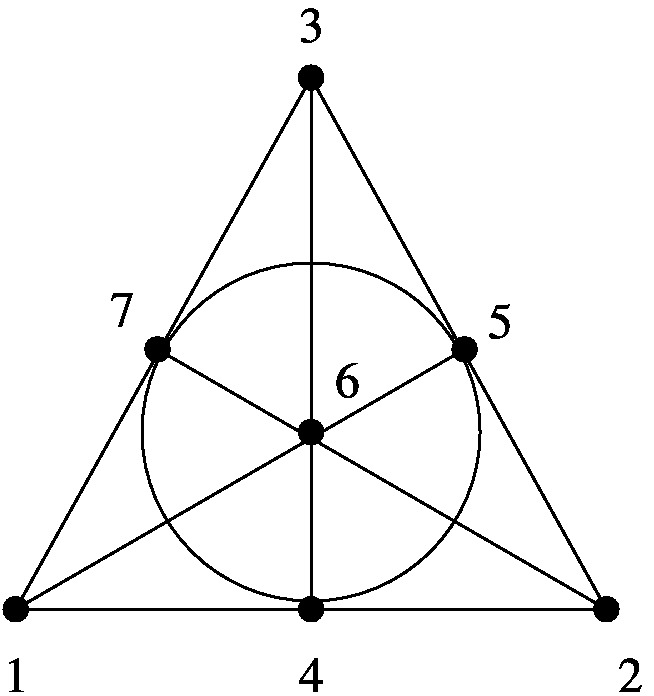}
  \caption{Fano plane of order $2$}
\end{center}
\end{figure}

Let $\underline{u^*}$ denote the marked element in $Q_7$ and $S(\underline{u^*})$ denote the subset of the set of balls $\{B(\underline{v_i}, 3): i=1\ldots 7  \}$ for which $\underline{u^*}\in B(\underline{v_i}, 3)$. We show that any subset $S$ is squarely determined by the choice of $\underline{u^*}$. To this end, we study the connection between the support of  $\underline{u^*}$ and the cardinality of $S(\underline{u^*})$.

\begin{obs}\label{obsi}.
\begin{itemize}

\item[(0)]  $|supp(\underline{u^*})|=0 \rightarrow |S(\underline{u^*})|=7$.
\item[(1)]  $|supp(\underline{u^*})|=1 \rightarrow |S(\underline{u^*})|=3$, and $S(\underline{u^*})$ consists of the balls, whose centers' support contains $supp(\underline{u^*})$.
\item[(2)] $|supp(\underline{u^*})|=2 \rightarrow |S(\underline{u^*})|=5$,  and $S(\underline{u^*})$ consists of the balls, whose centers' support intersects  $supp(\underline{u^*})$.
\item[(3)] $|supp(\underline{u^*})|=3$ and $supp(\underline{u^*})$ corresponds to a line in the Fano plane, i.e.  $\underline{u^*}=v_j$  $\rightarrow  |S(\underline{u^*})|=1$  and $S(\underline{u^*})=\{B(\underline{v_j}, 3)\}$.
\item[(3')] $|supp(\underline{u^*})|=3$ and $supp(\underline{u^*})$ does not correspond to any line in the Fano plane, $\rightarrow  |S(\underline{u^*})|=3$  and  $S(\underline{u^*})$ consists of the balls, whose centers' support intersection with  $supp(\underline{u^*})$ has cardinality $2$.  (Hence, the intersection of the supports of all these $3$ centers is the empty set.)
\end{itemize}
\end{obs}

\begin{claim} \label{sum7}
$S(\underline{u^*}) \dot\bigcup S(\underline{\overline{u^*}})= \{B(\underline{v_i}, 3): i=1\ldots 7  \}$.
\end{claim}

\begin{proof} $d(\underline{\overline{u}}, \underline{v})+d(\underline{{u}},\underline{v})=7$ for all $\underline{u}, \underline{v} \in Q_7$, thus exactly one of the summands is less then or equal to $3$.
\end{proof}

\begin{cor}\label{even}
$|supp(\underline{u^*})|>3 \rightarrow |S(\underline{u})|$ is even.
\end{cor}

To complete the proof of Lemma \ref{Fano}, it is enough to verify that no set $S\subseteq \{B(\underline{v_i}, 3): i=1\ldots 7  \}$ , $|S|$  odd, is assigned
 to different elements of  $Q_7$. Indeed, the other case $S\subseteq \{B(\underline{v_i}, 3): i=1\ldots 7  \}$ , $|S|$  even  will follow from this in view of Claim  \ref{sum7}, since $S(\underline{u^*})=S(\underline{u'})$ is equivalent to $S(\underline{\overline{u^*}})=S(\underline{\overline{u'}})$, where $|S(\underline{\overline{u^*}})|=7-|S(\underline{u^*})|$ is odd if $|S(\underline{u^*})|$ is even.\\
 But this verification is straightforward according to the case analysis of Observation \ref{obsi}, as it is easy to see how the sets $S\subseteq \{B(\underline{v_i}, 3): i=1\ldots 7  \}$  are determined accordingly in the cases $|S|=1, 3, 5$ and $7$.

\end{proof}

\section{The Erd\H os-R\'enyi random graph}
\label{sec:rand}

In this section we consider the Erd\H os-R\'enyi random graph $G(n,p)$. By $G(n,p)$ we mean the probability space
of all labeled graphs on $n$ vertices, where every edge appears
randomly and independently with probability $p=p(n)$. We say that
$G(n,p)$ possesses a property $\mathcal{P}$ \textit{asymptotically almost
surely}, or a.a.s. for brevity,
if the probability that $G(n,p)$ satisfies $\mathcal{P}$ tends to 1
as $n$ tends to infinity. We will find bounds for the quantities $M(G(n,p), r)$, $M(G(n,p), \leq r)$, and $M(G(n,p))$ which hold asymptotically almost surely. All our searches will use balls of equal radius, so the bounds we find for $M(G(n,p), \leq r)$ will hold for $M(G(n,p), r)$ as well.

Before we start to investigate the parameter $M(G(n,p),\leq r)$ generally, let us consider the special case when $r=1$. In this case the asymptotics of $M(G(n,p),\leq r)$ follow immediately from a result of Frieze, Martin, Moncel, Ruszink\'o, and Smyth \cite{FMMRS}. As we mentioned in the introduction, for any graph $G$ we always have $M(G,1) \le i^{(1)}_1(G)$.  The parameter $i^{(1)}_1(G(n,p))$ was studied by Frieze, Martin, Moncel, Ruszink\'o, and Smyth  who proved the following theorem.

\begin{theorem}[Frieze, Martin, Moncel, Ruszink\'o, Smyth \cite{FMMRS}]
\label{thm:fr} Let $q=p^2+(1-p)^2$ and assume $p, 1-p \ge \frac{\log \log n}{\log n}$. Then asymptotically almost surely $i^{(1)}_1(G(n,p))=(1+o(1))\frac{2 \log n}{\log (1/q)}$ holds.
\end{theorem}

It is quite natural to think that apart from some exceptional graphs, querying balls of radius 0 (i.e. asking whether the center is the unknown vertex or not) is not efficient and therefore $i^{(1)}_1(G)=M(G,\leq 1)$ or at least $i^{(1)}_1(G)=(1-o(1))M(G,\leq 1)$ holds. It is indeed the case for $G(n,p)$ and the proof of Lemma 2 in \cite{FMMRS} suffices. Frieze, Martin, Moncel, Ruszink\'o, and Smyth prove an upper bound on the probability that there exists an identifying code $C$ of size $\frac{(2-\varepsilon) \log n}{\log (1/q)}$ by considering the probability that there exists a pair of vertices in $V(G(n,p)) \setminus C$ that is not separated by the closed neighborhoods (i.e. 1-balls) of vertices of $C$. Since the probability that a set $B_1$ of 1-balls and a set $B_0$ of 0-balls together form a separating system is not larger than the probability that the closed neighborhoods of the centers $C$ of the balls in $B_1 \cup B_0$ separate all pairs of vertices in $V(G(n,p))\setminus C$, we obtain the following corollary of \tref{fr}.

\begin{cor}
Let $q=p^2+(1-p)^2$ and assume $p, 1-p \ge \frac{\log \log n}{\log n}$. Then almost surely $M(G(n,p),1)=(1+o(1))\frac{2 \log n}{\log (1/q)}$ holds.
\end{cor}

Let us now consider the problem for general $r$. Note that all our constructions use balls of equal radius thus all our bounds will apply to $i^{(1)}_r(G(n,p))$ as well. To obtain our bounds on $M(G(n,p),\leq r)$ and $M(G(n,p))$ we will need two tools. The first tool is the following theorem of Katona.

\begin{theorem} [Katona \cite{K}]
\label{thm:katona} Let $X$ be an $M$-element set and $\cA \subseteq 2^X$ be a separating system of subsets of $X$ such that for all $A \in\cA$ we have $|A| \le m$ for some integer $m<M/2$. Then the following inequality holds
\[
|\cA| \ge \frac{M}{m}\frac{\log M}{\log \frac{eM}{m}}.
\]
\end{theorem}
Secondly, we will apply the Chernoff inequalities (see \cite{AS}).

\noindent \textbf{Chernoff bound:} \textit{If $X$ is a binomially
distributed random variable with parameters $n$ and $p$ and
$\lambda=np$, then for any $t\geq 0$ we have
$$\mathbb{P}(X \geq \mathbb{E}X+t) \leq \exp\left(-\frac{t^2}{2(\lambda+t/3)}\right)$$
and $$\mathbb{P}(X \leq \mathbb{E}X-t) \leq
\exp\left(-\frac{t^2}{2\lambda}\right)\ .$$}

Our strategy will be quite simple. Whenever we are allowed to use balls of radius at most $r$, we will indeed use balls of radius $r$ and place the centers of the balls randomly. Then we have to analyze how many randomly placed balls will guarantee the separation of all pairs of vertices. Note that in any graph $G$ two vertices $x,y \in V(G)$ are separated by a ball of radius $r$ with center $z \in V(G)$ if and only if $z \in B_G(x,r) \triangle B_G(y,r)$, where $X\triangle Y=X\setminus Y\cup Y\setminus X$ for any two sets $X$ and $Y$. Therefore to obtain our results we will need bounds on $|B(x,r)|$, the sizes of all balls of radius $r$ in $G(n,p)$ and also on the sizes of pairwise intersections of balls of radius $r$ in $G(n,p)$.

These required bounds are strongly connected to the results of Bollob\'as. In \cite{B}, he determined the diameter $d=d(n)$ of $G(n,p)$ (it was also discovered by Klee and Larmann \cite{KL} when $d(n)$ is constant). Clearly, in any graph, it is only worth trying to search with balls of radius at most $d-1$. First we state his results on the diameter of $G(n,p)$, then we summarize his proof approach and also cite Lemma 5 of his paper as we will need to use that as well.

\begin{theorem}[Bollob\'as, \cite{B}]\label{thm:boll}

\textbf{(i)} Suppose $p^2n - 2 \log n \rightarrow \infty$ and $n^2(1-p)\rightarrow \infty$. Then $G(n,p)$ has diameter 2 asymptotically almost surely.

\textbf{(ii)} Suppose the functions $d = d(n) \ge 3$ and $0 <p = p(n) < 1$
satisfy
$(\log n)/d - 3 \log \log n \rightarrow \infty$,
$p^dn^{d-1} - 2 \log n \rightarrow \infty$ and $p^{d-1}n^{d-2} - 2 \log n\rightarrow -\infty$.
Then $G(n,p)$ has diameter $d$ asymptotically almost surely.
\end{theorem}

Actually, Bollob\'as proved more: he showed that if $p^dn^{d-1}=\log (n^2/c)$ for some positive constant $c$, then the distribution of the number of pairs of vertices in $G(n,p)$ that are of distance $d+1$ from each other tends to a Poisson distribution and therefore the probability that $G(n,p)$ has diameter $d$ tends to $e^{-c/2}$ while the probability that the diameter is $d+1$ tends to $1-e^{-c/2}$.

The main lemma which essentially deals with the required bounds on $|B_{G(n,p)}(x,r)|$ and $|B_{G(n,p)}(x,r) \triangle B_{G(n,p)}(y,r)|$ is Lemma 5 in the paper of Bollob\'as. He carries out all calculations only for $p^dn^{d-1}=\log (n^2/c)$, but when $p$ is bounded away from this threshold value, all calculations are even easier. Let us summarize how Bollob\'as obtains his bounds. For a fixed vertex $x$, we write $B(x,j)=B_{G(n,p)}(x,j)$ and $S(x,j)=B(x,j) \setminus B(x,j-1)$.
\begin{itemize}
\item
To obtain bounds on the size of $B(x,r)$ for a fixed vertex $x$ of $G(n,p)$ Bollob\'as exposes the edges of $G(n,p)$ in rounds. In the $j$th round he exposes edges between $S(x,j-1)$ and $V(G(n,p)) \setminus B(x,j-1)$ to determine $S(x,j)$ and thus $B(x,j)$. Note that if we 
condition on $B(x,j-1)$, then the size of $S(x,j)$ is a binomial random variable as for each vertex $y$ of $V(G(n,p)) \setminus B(x,j-1)$ the probability of $y \in S(x,j)$ is $1- (1-p)^{|S(x,j)|}$ and these events are independent. Careful analysis of concentration inequalities for binomial random variables gives a very small error probability of $|B(x,r)|$ not behaving as expected. Then the union bound is used to get the result below for all vertices.
\item
Similarly to the steps above, fixing two vertices $x,y$ of $G(n,p)$ one can expose the edges in the following rounds. In round $j$ one exposes edges between $S(x,j-1)$ and $V(G(n,p)) \setminus B(x,j-1)$ and also between $S(y,j-1)$ and $V(G(n,p)) \setminus B(y,j-1)$. Vertices in $B(x,j)\cap B(y,j) \setminus (B(x,j-1)\cap B(y,j-1))$ may come either from $B(x,j-1)\setminus B(y,j-1)$ or $B(y,j-1)\setminus B(x,j-1)$ or from $V(G(n,p)) \setminus (B(x,j-1)\cup B(y,j-1))$. Thus the number of new vertices in the intersection will be the sum of 3 binomially distributed random variable. Concentration inequalities and the union bound is used again to obtain bounds for all pairs $x,y$.
\end{itemize}

\begin{lemma}[Bollob\'as, Lemma 5 in \cite{B}]\label{lem:bollem}
Let $d$ denote the diameter of $G(n,p)$ with $2 \le d$ and $(\log n)/d - 3 \log \log n \rightarrow \infty$. Then the following holds for $G(n,p)$ asymptotically almost surely: for any $1 \le j \le d-2$ and any pair $x,y$ of vertices we have $$|B_{G(n,p)}(x,j)|=(1+o(1))(np)^j \hskip 0.4truecm\textnormal{and}\ \hskip 0.3truecm
|B_{G(n,p)}(x,j) \cap B_{G(n,p)}(y,j)|=O(n^{2j-1}p^{2j}).$$
\end{lemma}

\lref{bollem} enables us to obtain bounds on $M(G(n,p),r)$ when $r\le d-2$.

\begin{theorem}\label{thm:random}
 Suppose the functions $d = d(n) \ge 3$ and $0 <p = p(n) < 1$
satisfy
$$(\log n)/d - 3 \log \log n \rightarrow \infty,\hskip 0.1truecm p^dn^{d-1} - 2 \log n \rightarrow \infty, \hskip 0.1truecm p^{d-1}n^{d-2} - 2 \log n\rightarrow -\infty.$$
Then for any $1 \le r\le d-2$ we have that
$$(1-o(1))n^{1-r}p^{-r}\frac{\log n}{\log (n^{1-r}p^{-r})}\leq M(G(n,p),\leq r)\le (1+o(1)) n^{1-r}p^{-r}\ln n$$
holds asymptotically almost surely.
\end{theorem}

\begin{proof}
To obtain the lower bound we have to plug the values $M=n$ and $m=(1+o(1))(np)^r$ into \tref{katona}. Note that the value of $m$ follows from \lref{bollem}. Then \tref{katona} yields
\[
M(G(n,p),r) \ge \frac{\log n}{\log e\frac{n}{(1+o(1))(np)^r}}\frac{n}{(1+o(1))(np)^r}=(1-o(1))n^{1-r}p^{-r}\frac{\log n}{\log (n^{1-r}p^{-r})}.
\]

To obtain the upper bound we will place balls of radius $r$ at centers $u_1,u_2,...,u_t \in V(G(n,p))$ chosen uniformly at random and independently of each other. We introduce the random indicator variables $Y_{v_i,v_j}$ of the event that the vertices $v_i,v_j$ are not separated by the balls $B_{G(n,p)}(u_1,r),\dots,B_{G(n,p)}(u_t,r)$. We write $Y=\sum_{1\le i<j\le n}Y_{v_i,v_j}$ to denote the random variable of the total number of unseparated vertices. Observe that by \lref{bollem} we have
$|B_{G(n,p)}(v_i,r)\setminus B_{G(n,p)}(v_j,r)|=(1+o(1))(np)^r$ and thus
$$\mathbb{E}(Y_{v_i,v_j})=\left(1-\frac{|B_{G(n,p)}(v_i,r)\mathbin{\triangle} B_{G(n,p)}(v_j,r)|}{n}\right)^t=\left(1-\frac{(2+o(1))(np)^r}{n}\right)^t.$$
Therefore if $t\ge n^{1-r}p^{-r}\ln n$ holds, then we have
\[
\mathbb{E}(Y)\le \binom{n}{2}\left(1-\frac{(2+o(1))(np)^r}{n}\right)^t\le \frac{1}{2}\exp(2\ln n-2tn^{r-1}p^r) \le \frac{1}{2}.
\]
Therefore there exists a particular choice of $u_1,u_2,...,u_t$ such that the number of non-separated pairs of vertices is 0.
\end{proof}

One would expect that $|B_{G(n,p)}(x,d-1)|=(1+o(1))(pn)^{d-1}$ should hold. It does for most values of $p$, but for a tiny range it fails since if $n^{-(d-2)/(d-1)} \le p\le n^{-(d-2)/(d-1)}(2\log n)^{1/(d-1)}$, then $(pn)^{d-1}$ is larger than $n$ and still the diameter of $G(n,p)$ is~$d$. For this range of $p$  we introduce the function $f(n)=f_d(n)$ defined by the equation $(np)^{d-1}=n\cdot f(n)$. We will subdivide the interval of $p$ for which the diameter of $G(n,p)$ is $d$ into four smaller intervals.  When $f(n) \rightarrow \infty$, then balls of radius $d-1$ consist of almost all vertices, and therefore we will be interested in their complement and in the fact that whether these complements are still bigger than the balls of radius $d-2$. If this is so, then it is still worth picking randomly centered balls of radius $d-1$, than those of radius $d-2$.

\begin{lemma}
\label{lem:randomball} (i) If $(np)^{d-1}=o(n)$, then a.a.s $$|B(x,d-1)|=(1+o(1))(np)^{d-1}, |B(x_1,d-1) \triangle B(x_2,d-1)|=(2+o(1))(np)^{d-1}$$ hold for all vertices $x,x_1,x_2 \in V(G(n,p))$.

(ii) If $(np)^{d-1}=(c+o(1))n$ holds for some positive constant $c$, then a.a.s. $$|\overline{B(x,d-1)}|=(1-e^{-c}+o(1))n, |B(x_1,d-1) \triangle B(x_2,d-1)|=(1-e^{-2c}+o(1))n/e^{f(n)}$$ hold for all vertices $x,x_1,x_2 \in V(G(n,p))$.

(iii) If $f(n)\le (1-\varepsilon)\frac{1}{d-1}\log n$ holds, then a.a.s.$$|\overline{B(x,d-1)}|=(1+o(1))n/e^{f(n)}, |B(x_1,d-1) \triangle B(x_2,d-1)|=(2+o(1))n/e^{f(n)}$$ hold for all vertices $x,x_1,x_2 \in V(G(n,p))$.

(iv) If $f(n)\ge (1+\varepsilon)\frac{1}{d-1}\log n$ holds, then a.a.s. we have $$|\overline{B(x,d-1)}|<|B(x,d-2)|$$ for all $x \in V(G(n,p)$.
\end{lemma}

\begin{proof} Let us condition on the event that is stated in \lref{bollem}. More precisely, when proving statements about $B(x,d-1)$, we pick a vertex $x$, expose edges in rounds among vertices $S(x,j-1)$ and $V(G(n,p))\setminus B(x,j-1)$. We condition on the event $E_x$ that $|S(x,d-2)|=(1+o(1))(np)^{d-2}$. Bollob\'as proved that $E_x$ holds with probability $1-n^{-K}$ for some large constant $K$ and thus $E_x$ holds with high probability for all vertices simultanously. Similarly, when proving statements on $|B(x_1,d-1)\cap B(x_2,d-1)|$ we pick two vertices $x_1$ and $x_2$, 
expose edges in rounds among vertices in $S(x_1,j-1)$ and $V(G(n,p))\setminus B(x_1,j-1)$ and also among vertices in $S(x_2,j-1)$ and $V(G(n,p))\setminus B(x_2,j-1)$. We condition on the event $E_{x_1,x_2}$ that $|S(x_1,d-2)|=(1+o(1))(np)^{d-2}=|S(x_2,d-2)|$ and $|B(x_1,d-2)\cap B(x_2,d-2)|=O((np)^{d-3})$. Again, these events hold with probability $1-n^{-K}$ for some large constant $K$.

Before considering the four cases let us observe that $\Omega(n^{1/2})=(np)^{d-2}=o(n)$. We will also use the inequalities $pa-p^2\binom{a}{2} \le 1-(1-p)^{a}\le pa$.

Consider first the cases (i), (ii) when $(np)^{d-1}=o(n)$ or $(np)^{d-1}=(c+o(1))n$ holds for some positive constant $c$. In this case for a vertex $y \notin B(x,d-2)$ we have $\mathbb{P}(y \in S(x,d-1))=1-(1-p)^{|S(x,d-2)|}$. This is
$(1+o(1))p^{d-1}n^{d-2}$ and thus $\mathbb{E}(|S(x,d-1)|)=(1+o(1))(np)^{d-1}$ if $(np)^{d-1}=o(n)$, while $(n-o(n))(1-(1-p)^{|S(x,d-2)|})=(1-e^{-c}+o(1))n$ if $(np)^{d-1}=(c+o(1))n$. As $(np)^{d-1}\ge (np)^{d-2}=\Omega(n^{1/2})$ holds, by Chernoff's inequality we obtain that $\mathbb{P}(|S(x,d-1)-\mathbb{E}(|S(x,d-1)|)|\ge n^{1/3})$ is exponentially small in a positive power of $n$. The union bound gives that asymptotically almost surely $|B(x,d-1)|=(1+o(1))|S(x,d-1)|=\mathbb{E}(|S(x,d-1)|)$ holds for all vertices.  

Let us condition on the sets $B(x_1,d-2),S(x_1,d-2),B(x_2,d-2)$ and $S(x_2,d-2)$. Then a vertex $y$ can become an element of $B(x_1,d-1)\cap B(x_2,d-1)$ in the following ways: 

(a) either $y \notin B(x_1,d-2) \cup B(x_2,d-2)$  and (a1) it is connected to a vertex in $B(x_1,d-2) \cap B(x_2,d-2)$ or (a2) is connected to at least one vertex both from $S(x_1,d-2)$ and $S(x_2,d-2)$ or 

(b) $y \in S(x_1,d-2) \setminus B(x_2,d-2)$ and $y$ is connected to a vertex in $S(x_2,d-2)$ or vice versa. The number of such vertices is binomially distributed in all three cases. 

In case (a1), using the result of \lref{bollem} on $|B(x_1,d-2)\cap B(x_2,d-2)|$ the expected value of such vertices is $O(p^{2(d-2)+1}n^{2(d-2)})=O(\frac{p}{(pn)^2}((pn)^{d-1})^2)=O(\frac{pn}{(pn)^2}(pn)^{d-1})=o((pn)^{d-1})$. 

In case (a2), the expected value of such vertices is $(1-(1-p)^{|S(x_1,d-2)|})(1-(1-p)^{|S(x_2,d-2)|})$ which is $O(\frac{1}{n}((pn)^{d-1})^2)=o((pn)^{d-1})$ provided $(np)^{d-1}=o(n)$. If $(np)^{d-1}=(c+o(1))n$, then this value is $\big{(}(1-e^{-c})^2+o(1)\big{)}n$. 

Finally, in case (b), the expected value of such vertices is $O((pn)^{d-2}p^{d-1}n^{d-2})=o((pn)^{d-1})$. In all three cases Chernoff's inequality and the union bound can be applied to obtain the same bounds not only for one, but all pairs of vertices.

Altogether, we obtain that $|B(x_1,d-1) \triangle B(x_2,d-1)|=(2+o(1))(np)^{d-1}$ holds almost surely for all $x_1,x_2$ if $(np)^{d-1}=o(n)$ and that
$|B(x_1,d-1) \triangle B(x_2,d-1)|=[2(1-e^{-c})-(1-e^{-c})^2+o(1)]n=(1-e^{-2c}+o(1))n$ if $(np)^{d-1}=(c+o(1))n$ holds.

Let us now turn our attention to (iii) and (iv). In these cases we are interested in $|\overline{B(x,d-1)}|$. Let us repeat that $f(n)$ is defined by
$f(n)=f_d(n)=(np)^{d-1}=n\cdot f(n)$ and that $f(n)$ tends to infinity as otherwise we were in case (i) or (ii).
If we condition on $B(x,d-2)$ and $S(x,d-2)$, then the size of $\overline{B(x,d-1)}$ is binomially distributed with $\mathbb{E}(|\overline{B(x,d-1)}|)=(n-|B(x,d-2)|)(1-p)^{|S(x,d-2)|}=(1+o(1))e^{\log n -(1+o(1))f(n)}$. If $f(n)\ge (1+\varepsilon)\frac{1}{d-1}\log n$ holds for some positive $\varepsilon$, then this expression is $o((np)^{d-2})$. Chernoff's inequality and the union bound ensures that this holds for all $x$ and we are done by (iv).

Using Chernoff's inequality we obtain that as long as $f(n)\le (1-\varepsilon)\frac{1}{d-1}\log n$ holds, we have a.a.s $|\overline{B(x,d-1)}|=(1+o(1))n/e^{f(n)}>|B(x,d-2)|$ for any vertex $x$. As $I:=\overline{B(x_1,d-1)}\cap \overline{B(x_2,d-1)}=\overline{B(x_1,d-1)\cup B_(x_2,d-1)}$ we have that for any $y\in V(G(n,p))\setminus (B(x_1,d-2)\cup B(x_2,d-2))$ the probability that $y$ is in $I$ is $(1-p)^{|S(x_1,d-2)\cup S(x_2,d-2)|}$ and thus $\mathbb{E}(|I|)=(1+o(1))n/e^{2f(n)}$. Note that as $f(n)$ tends to infinity, we have $n/e^{2f(n)}=o(n/e^{f(n)})$. The usual Chernoff's inequality argument shows that $|\overline{B(x_1,d-1)}\cap \overline{B(x_2,d-1)}|=o(n/e^{f(n)})$ and thus $|\overline{B(x_1,d-1)}\triangle \overline{B(x_2,d-1)}|=(2+o(1))n/e^{f(n)}$ holds a.a.s for any $x_1,x_2 \in V(G(n,p))$.
\end{proof}

\medskip

The results of \lref{randomball} enable us to imitate the proof of \tref{random} to obtain bounds on $M(G(n,p))$. Part (iv) of \lref{randomball} shows that in that range of $p$, one cannot improve the bounds of \tref{random} by using random balls of radius $d-1$. In all other cases, replacing the values of $|B(x_1,d-1) \triangle B(x_2,d-1)|=|\overline{B(x_1,d-1)}\triangle \overline{B(x_2,d-1)}|$ does indeed improve on bounds obtained by placing balls of radius $d-2$. However, their formulation would be a little lengthy due to the fact that sometimes one has to calculate with the size of balls of radius $d-1$ and sometimes with the size of their complements. Therefore here we state only the simplest case (corresponding to case (i) in Lemma 1.7.) and just mention that in all cases the lower and upper bounds that can be obtained differ only by a factor of a power of $\log n$.

\begin{theorem}\label{thm:randmore}
Suppose the functions $d = d(n) > 3$ and $0 <p = p(n) < 1$
satisfy
$(\log n)/d - 3 \log \log n \rightarrow \infty$,
$p^dn^{d-1} - 2 \log n \rightarrow \infty$. Furthermore,
 if $(np)^{d-1}=o(n)$ holds, then a.a.s we have
$$(1-o(1))\frac{n}{(np)^{d-1}}\frac{\log n}{\log n/(np)^{d-1}}\le M(G(n,p))\le (1+o(1)) \frac{n}{(np)^{d-1}}\ln n.$$
\end{theorem}

\section{Graphs of bounded maximum degree}
\label{sec:maxdegsec}

The aim of this section is to prove \tref{maxdeg}, that is for any $G \in \cG_{\Delta,n}$ we have $A(G)=\Theta_{\Delta}(\log n)$ where $\cG_{\Delta,n}$ denotes the set of finite graphs on $n$ vertices with maximum degree at most $\Delta$. Note that the lower bound $\log n\le A(G)$ holds for any graph as the sequence of answers must differ for all vertices of $G$. The idea behind \tref{maxdeg} is the following: for every $G \in \cG_{\Delta}$, $v \in V$ and $r\geq 1$  we have $|B_G(v,r+1)|/|B_G(v,r)| \le \Delta-1$, and this enables us to imitate the halving argument of adaptive binary search. This is formulated in the following lemma.

\begin{lemma}\label{lem:predelta}

Let  $G$ be a connected graph with maximum degree $\Delta$. For every $X \subseteq V(G)$ with $|X| \geq \Delta+2$, there exists a vertex $v(X) \in V(G)$ and a natural number $r(X)$ satisfying $$ \frac{1}{\Delta+2}|X| \leq |B(v(X),r(X)) \cap X | \leq \frac{\Delta+1}{\Delta+2}|X|.$$

\end{lemma}

\begin{proof} For all $w \in V(G)$ and $X \subseteq V(G)$ we define the following function: $$ f(w, X) := \min \left\{r : | B(w,r) \cap X | \geq \frac{\Delta+1}{\Delta+2}|X| \right\}.$$

Since $G$ is connected,  $f(w,X)$ is well defined for all $w \in V(G)$.  Also observe that $f(w,X) \geq 2$, since $\frac{\Delta+1}{\Delta+2}|X| \geq \Delta+1$.

Define $m(X)=\min\{f(w,X) : w \in V(G)\}$, and let $w(X)$ be any vertex satisfying $f(w(X),X)=m(X)$.
We claim that we can choose one of the neighbours of $w(X)$ as our vertex $v(X)$.
Suppose, for the sake of contradiction, that none of the neighbors of $w(X)$ fulfill the statement of the lemma. By the definition of $m(X)$, we know that $ | B(u, m(X)-1) \cap X | \leq \frac{1}{\Delta+2}|X|$ holds for all vertices $u \in V(G)$. But we also know that $$B(w(X),m(X)) \subseteq \bigcup_{\{u,w(X)\} \in E(G)} B(u, m(X)-1).$$
This gives us
$$X \cap B(w(X),m(X)) \hskip 0.5truecm \subseteq \hskip 0.5truecmX \cap \bigcup_{\{u,w(X)\} \in E(G)} B(u, m(X)-1).$$
This gives a contradiction since $|X \cap B(w(X),m(X))| \ge  \frac{\Delta+1}{\Delta+2}|X|$ holds by the definition of $m(X)$, whereas $|X \cap \bigcup_{\{u,w(X)\} \in E(G)} B(u, m(X)-1)| \le \sum_{\{u,w(X)\} \in E(G)} |B(u, m(X)-1)| \le \frac{\Delta}{\Delta+2}|X|$.

\end{proof}

Now we are ready to prove \tref{maxdeg}.
\begin{proof}[Proof of \tref{maxdeg}:] Let $G \in \cG_{\Delta,n}$ be a connected graph and set $X_0=V(G)$. For $i>0$ we repeatedly apply \lref{predelta} to $X_{i-1}$ as long as $|X_{i-1}| \geq \Delta+2$ holds. We ask the query $B(v(X_{i-1}), r(X_{i-1}))$ and let

$$X_i:=
\begin{cases}
X_{i-1} \cap B(v(X_{i-1}), r(X_{i-1})) &\text{if the answer is yes;} \\
X_{i-1} \setminus B(v(X_{i-1}), r(X_{i-1})) &\text{if the answer is no.}
\end{cases}$$

Clearly the unknown vertex is contained in $X_i$ for all values of $i$. If $|X_{i-1}| \leq \Delta+2$, then we ask its vertices one by one (more precisely the balls of radius 0 around its vertices).
The length of this process is at most $$\log_{\frac{\Delta+3}{\Delta+2}}n + \Delta+1.$$ This completes the proof of the theorem.
\end{proof}
By the same argument, one can prove a version of Theorem~\ref{thm:maxdeg} where the upper bound $\Delta(n)$ on the maximum degree in the graph is a function which grows with $n$.  Namely, we  obtain a $\log_{\frac{\Delta(n)+3}{\Delta(n)+2}}n + \Delta(n) +2$ upper bound on the length of an adaptive search.

\bibliography{hypsearch}

\begin{thebibliography}{10}

\bibitem{AS}
N.~Alon and J.~Spencer.
\newblock {\em The probabilistic method}.
\newblock Wiley-Interscience., 1991.

\bibitem{BGLM}
Y.~Ben-Haima, S.~Gravierb, A.~Lobsteinc, and J.~Moncel.
\newblock Adaptive identification in graphs.
\newblock {\em J. Combin. Theory Ser. A}, 115(7):1114--1126, 2008.

\bibitem{B}
B.~Bollob\'as.
\newblock The diameter of random graphs.
\newblock {\em Trans. Amer. Math. Soc.}, 267:41--52, 1981.

\bibitem{CCHL}
I.~Charon, G.~Cohen, O.~Hudry, and A.~Lobstein.
\newblock New identifying codes in the binary hamming space.
\newblock {\em European J. Combin.}, 31:491--501, 2010.

\bibitem{CHLL}
G.~Cohen, I.~Honkala, S.~Litsyn, and A.~Lobstein.
\newblock {\em Covering Codes}.
\newblock North-Holland Publishing Co., 1997.

\bibitem{EJLR1}
G.~Exoo, V.~Junnila, T.~Laihonen, and S.~Ranto.
\newblock Upper bounds for binary identifying codes.
\newblock {\em Advances in Applied Mathematics}, 42(3):277--289, 2009.

\bibitem{EJLR2}
G.~Exoo, V.~Junnila, T.~Laihonen, and S.~Ranto.
\newblock Improved bounds on identifying codes in binary {H}amming spaces.
\newblock {\em European J. Combin.}, 31(3):813--827, 2010.

\bibitem{ELR}
G.~Exoo, T.~Laihonen, and S.~Ranto.
\newblock New bounds on binary identifying codes.
\newblock {\em Discrete Applied Mathematics}, 156(12):2250--2263, 2008.

\bibitem{FP}
F.~Foucaud and G.~Perarnau.
\newblock Bounds for identifying codes in terms of degree parameters.
\newblock {\em Electron. J. Combin.}, 19, 2012.

\bibitem{FMMRS}
A.~Frieze, R.~Martin, J.~Moncel, M.~Ruszink\'o, and C.Smyth.
\newblock Codes identifying sets of vertices in random networks.
\newblock {\em Discrete Math.}, 307:1094--1107, 2007.

\bibitem{HM}
F.~Harary and R.~Melter.
\newblock On the metric dimension of a graph.
\newblock {\em Ars Combinatoria}, 2:191--195, 1976.

\bibitem{JL}
S.~Janson and T.~Laihonen.
\newblock On the size of identifying codes in binary hypercubes.
\newblock {\em J. Combin. Theory Ser. A}, 116(5):1087--1096, 2009.

\bibitem{J}
V.~Junnila.
\newblock Adaptive identication of sets of vertices in graphs.
\newblock {\em Discrete Math. Theor. Comput. Sci.}, 14(1):69--86, 2012.

\bibitem{KCL}
M.~G. Karpovsky, K.~Chakrabarty, and L.~B. Levitin.
\newblock On a new class of codes for identifying vertices in graphs.
\newblock {\em IEEE Trans. Inform. Theory}, 44:599--611, 1998.

\bibitem{K}
G.~Katona.
\newblock On separating systems of a finite set.
\newblock {\em J. Combin. Theory}, 1(2):174--194, 1966.

\bibitem{eml}
G.~Katona.
\newblock Combinatorial search problems, problem booklet of the 5th
  eml\'ekt\'abla workshop,
  http://www.renyi.hu/$\sim$emlektab/index$_{-}$booklet.html.
\newblock 2013.

\bibitem{KL}
V.~L. Klee and D.~G. Larmann.
\newblock Diameters of random graphs.
\newblock {\em Canad. J. Math.}, 33:618--640, 1981.

\bibitem{KSV}
M.~Krivelevich, B.~Sudakov, and V.~H. Vu.
\newblock Covering codes with improved density.
\newblock {\em IEEE Trans. Inform. Theory}, 49:1812--1815, 2003.

\bibitem{L}
A.~Lobstein.
\newblock http://www.infres.enst.fr/$\sim$ lobstein/debutbibidetlocdom.pdf.

\bibitem{M}
J.~Moncel.
\newblock Codes identifiants dans les graphes.
\newblock {\em These de Doctorat, Universit\'e de Grenoble, France}, 2005.

\bibitem{S}
P.~Slater.
\newblock Leaves of trees.
\newblock {\em Congr. Numer.}, 14:549--559, 1975.

\end{thebibliography}
\bibliographystyle{abbrv}

\end{document}